\documentclass[11pt]{amsart}

\usepackage[margin=1in]{geometry}
\usepackage{amsmath,amssymb,amsthm,mathtools}
\usepackage{enumitem}
\usepackage{microtype}
\usepackage[hidelinks]{hyperref}
\newtheorem{theorem}{Theorem}[section]
\newtheorem{lemma}[theorem]{Lemma}
\newtheorem{proposition}[theorem]{Proposition}
\newtheorem{corollary}[theorem]{Corollary}
\theoremstyle{remark}

\newcommand{\C}{\mathbb C}
\newcommand{\D}{\mathbb D}
\newcommand{\e}{\mathrm e}
\newcommand{\norm}[1]{\left\lVert #1\right\rVert}
\newcommand{\RePart}{\operatorname{Re}}
\newcommand{\ImPart}{\operatorname{Im}}
\newcommand{\Pcal}{\mathcal P}
\title[Residual bounds for Schur-stable polynomials]{Residual bounds
for Schur-stable polynomials}

\author{Xiaojun Tan}
\address{Institute of Theoretical Physics, Chinese Academy of Sciences,
Beijing 100190, China}
\author{Qihang Wang}
\address{School of Mathematical Sciences, Peking University,
Beijing 100871, China}
\author{Wei Huang}
\address{RIKEN Center for Advanced Intelligence Project (AIP), Tokyo,
Japan, and The Institute of Statistical Mathematics, Tokyo, Japan}
\author{Kun Chen}
\address{Institute of Theoretical Physics, Chinese Academy of Sciences,
Beijing 100190, China}
\email{chenkun@itp.ac.cn}
\thanks{Xiaojun Tan and Qihang Wang contributed equally to this work.
Kun Chen is the corresponding author.}

\hypersetup{
  pdftitle={Residual bounds for Schur-stable polynomials},
  pdfauthor={Xiaojun Tan; Qihang Wang; Wei Huang; Kun Chen},
  pdfsubject={Residual bounds for Schur-stable polynomials},
  pdfkeywords={Schur-stable polynomials, polynomial residuals, zero
    separation, power sums, Erdos Problem 973}
}


\begin{document}

\begin{abstract}
Let \(\mathcal P_n\) be the class of degree-\(n\) polynomials \(P\)
satisfying \(P(0)=1\) whose zeros lie in the closed unit disk, and let
\(r_n\) be the infimum of
\[
 \frac{\lVert P'-P'(0)P\rVert_{H^2}}{\lVert P\rVert_{H^2}}
\]
over \(\mathcal P_n\).  We prove the quantitative residual bound
\[
 r_n\geq
 \exp\!\bigl(-(1+o(1))\sqrt n\log n\bigr)
 \qquad(n\to\infty).
\]
As an application, we answer Erd\H{o}s Problem 973 on exterior power
sums in the negative, in a form quantitatively stronger than the
answer first obtained by Luo, Yang, and Zhu.
\end{abstract}

\maketitle

\section{Introduction}

A polynomial of degree \(n\geq1\) cannot satisfy the
constant-coefficient equation \(P'=P'(0)P\), whose solutions are
exponentials.  This paper bounds from below how nearly a polynomial
with all zeros in the closed unit disk can satisfy it, and applies
the bound to a problem of Erd\H{o}s on power sums.

Write \(\D=\{w\in\C:|w|<1\}\), and let \(\Pcal_n\) be the class of
degree-\(n\) polynomials \(P\) such that
\[
 P(0)=1,
 \qquad
 \{\text{zeros of }P\}\subseteq\overline\D.
\]
Here, Schur-stable means that all zeros lie in the closed unit disk, so
zeros on the unit circle are allowed.  For
\(P\in\Pcal_n\), define the normalized residual
\[
 \rho(P)=
 \frac{\norm{P'-P'(0)P}_{H^2}}{\norm P_{H^2}},
 \qquad
 r_n=\inf_{P\in\Pcal_n}\rho(P).
\]
The numerator measures the failure of \(P\) to satisfy the
constant-coefficient equation \(P'=P'(0)P\).

\subsection*{Notation}

For a polynomial \(Q(t)=\sum_kq_kt^k\), the \(H^2\) norm is
\[
 \norm Q_{H^2}
 =\Bigl(\sum_k|q_k|^2\Bigr)^{1/2};
\]
by Parseval's identity it equals the mean
\(\bigl(\frac1{2\pi}\int_0^{2\pi}
|Q(\e^{i\theta})|^2\,d\theta\bigr)^{1/2}\) over the unit circle.
Only two norms occur in this paper: we abbreviate
\(\norm Q_2=\norm Q_{H^2}\), and we write
\(\norm Q_\infty=\max_{|\zeta|=1}|Q(\zeta)|\), which equals
\(\max_{|t|\leq1}|Q(t)|\) by the maximum-modulus principle.

Our main result is the following.

\begin{theorem}[Residual bound]\label{thm:stable}
For every \(\epsilon>0\), there exists \(N_\epsilon\in\mathbb N\)
such that
\[
 r_n\geq
 \exp\!\left[-\sqrt n\left(
 \log n+\log\frac2{\log2}+\epsilon\right)\right]
 \qquad(n\geq N_\epsilon).
\]
\end{theorem}

The proof of Theorem~\ref{thm:stable} occupies
Section~\ref{sec:proof}.  It is by contradiction: we suppose that
some \(P\in\Pcal_n\) satisfies \(\rho(P)\leq\e^{-c_nn}\), where
\(c_n\sim(\log n)/\sqrt n\) is slightly larger than the exponent
in the theorem, and proceed in four steps.

\begin{enumerate}[label=\emph{Step \arabic*.},leftmargin=*,
 itemsep=2pt,wide]
\item\emph{Normalization}
(\S\ref{subsec:normalization}).
Rescaling \(P\) at a boundary point where \(|P|\) is maximal
produces \(G\) with \(G(1)=\norm G_\infty=1\) and the same residual.
Thus \(G\) solves the constant-coefficient equation \(G'=\mu G\),
for an explicit constant \(\mu\), up to an error of \(H^2\) norm at
most \(\e^{-c_nn}\).
\item\emph{Zero separation}
(\S\ref{subsec:separation}).
Integrating this perturbed equation along the segment from each zero
\(\beta\) of \(G\) to the maximum point \(1\) forces
\(|1-\beta|\geq c_n/2\).  Consequently the reciprocals
\(v_\ell=1/(1-\beta_\ell)\) satisfy \(|v_\ell|\leq2/c_n\), and their
average \(m_1=\frac1n\sum_\ell v_\ell\) is real and at least
\(\frac12\).
\item\emph{Higher moments}
(\S\ref{subsec:moments}).
Differentiating the logarithmic derivative of \(G\) at the maximum
point bounds each moment \(m_k=\frac1n\sum_\ell v_\ell^k\),
\(2\leq k\leq n+1\), by an explicit multiple of
\(n^{k-2}\e^{-c_nn}\).
\item\emph{The polynomial test}
(\S\ref{subsec:test}).
Averaging \(\Phi(v)=v(1-\frac{c_n}2v)^{K-1}\), with \(K\) chosen
optimally, over \(v_1,\ldots,v_n\): by Steps~2--3 the result is
close to \(m_1\), yet the contraction
\(|1-\frac{c_n}2v_\ell|<1\) keeps it below \(m_1\) by a fixed
factor---a contradiction.  The choice of the parameters of \(\Phi\)
is a purely computational lemma, proved in
Appendix~\ref{app:asymptotics}.
\end{enumerate}

As an application, Theorem~\ref{thm:stable} yields a negative answer
to a problem of Erd\H{o}s~\cite[p.~213]{Erdos1965}, recorded as
Problem~7.3 in Hayman's collection~\cite{Hayman1974} and catalogued
as Erd\H{o}s Problem 973~\cite{Bloom973}.

\begin{corollary}[Erd\H{o}s Problem 973]\label{cor:973}
There is no constant \(C>1\) with the following property: for every
\(n\geq2\), there exist \(z_1,\ldots,z_n\in\C\) with \(z_1=1\),
\(|z_j|\geq1\) for \(1\leq j\leq n\), and
\[
 \max_{2\leq k\leq n+1}
 \left|\sum_{j=1}^n z_j^k\right|<C^{-n}.
\]
\end{corollary}

The negative answer was first obtained, by a different method, by
Luo, Yang, and Zhu~\cite{LuoYangZhu2026}.
Section~\ref{sec:applications} derives Corollary~\ref{cor:973} from a
quantitative, subexponential lower bound for the power-sum maximum
(Corollary~\ref{cor:power}).

\section{Proof of the residual bound}\label{sec:proof}

\begin{proof}[Proof of Theorem~\ref{thm:stable}]
Fix \(\epsilon>0\), and put
\[
 a=\log2,
 \qquad
 \sigma=\log\frac2a+\epsilon;
\]
this choice of \(\sigma\) is what Lemma~\ref{lem:test-asymptotics}
below requires in the final step.
Suppose for contradiction that the conclusion fails.  Then there
exist a strictly increasing sequence \(n_j\to\infty\) and polynomials
\(P_j\in\Pcal_{n_j}\) such that
\[
 \rho(P_j)<
 \exp\!\left[-\sqrt{n_j}(\log n_j+\sigma)\right].
\]
For notational simplicity, relabel \(n_j\) and \(P_j\) as \(n\) and
\(P\), respectively, and set
\[
 c_n=\frac{\log n+\sigma}{\sqrt n}.
\]
All asymptotic notation below is for \(n\to\infty\) along this
relabeled sequence, with \(\epsilon\) fixed.  We then have
\begin{equation}\label{eq:bad}
 \rho(P)\leq\e^{-c_n n}.
\end{equation}
We shall derive a contradiction.

Two elementary estimates are used repeatedly.  First, Bernstein's
inequality on the unit circle: if \(Q\) is a polynomial of degree at
most \(n\), then
\begin{equation}\label{eq:bernstein}
 \norm{Q^{(m)}}_\infty
 \leq n(n-1)\cdots(n-m+1)\norm Q_\infty
 \leq n^m\norm Q_\infty
 \qquad(0\leq m\leq n).
\end{equation}
Second, the coefficient bound: if \(Q(t)=\sum_{k=0}^nq_kt^k\) and
\(|w|\leq1\), then, by the Cauchy--Schwarz inequality applied to the
coefficient vector,
\begin{equation}\label{eq:coefficient-bound}
 |Q(w)|\leq\sum_{k=0}^n|q_k|
 \leq\sqrt{n+1}\,\norm Q_2.
\end{equation}

\subsection{Maximum-point normalization}\label{subsec:normalization}

Choose \(\tau\in\partial\D\) such that
\(|P(\tau)|=\max_{|t|=1}|P(t)|\), and define
\begin{equation}\label{eq:normalization}
 G(w)=\frac{P(\tau w)}{P(\tau)},
 \qquad
 \mu=\tau P'(0),
 \qquad
 E=G'-\mu G.
\end{equation}
By the maximum-modulus principle,
\(|P(\tau)|\geq|P(0)|=1>0\), so \(G\) is well defined, and
\begin{equation}\label{eq:normG}
 G(1)=1,
 \qquad
 \norm G_\infty=1,
 \qquad
 \norm G_2\leq1.
\end{equation}
The \(k\)-th coefficient of \(G\) is \(\tau^k/P(\tau)\) times the
\(k\)-th coefficient of \(P\), and the \(k\)-th coefficient of \(E\)
is \(\tau^{k+1}/P(\tau)\) times the \(k\)-th coefficient of
\(P'-P'(0)P\).  Since \(|\tau|=1\), the factors \(\tau^k\) and
\(\tau^{k+1}\) leave \(\norm\cdot_2\) unchanged, and the common
scalar \(P(\tau)\) cancels in the ratio, so
\begin{equation}\label{eq:residual-invariance}
 \frac{\norm E_2}{\norm G_2}=\rho(P),
 \qquad
 \norm E_2=\rho(P)\norm G_2\leq\e^{-c_n n}.
\end{equation}

Let \(\beta_1,\ldots,\beta_n\) be the zeros of \(G\), counted with
multiplicity.  They lie in \(\overline\D\), and none equals \(1\)
because \(G(1)=1\).  Put \(d=G'(1)\).  Since
\(\theta\mapsto|G(\e^{i\theta})|^2\) is smooth and attains its
maximum at \(\theta=0\), and since \(G(1)=1\),
\[
 0=\left.\frac{d}{d\theta}|G(\e^{i\theta})|^2\right|_{\theta=0}
 =2\RePart\bigl(iG'(1)\overline{G(1)}\bigr)
 =-2\ImPart G'(1),
\]
so
\begin{equation}\label{eq:d-real}
 d\in\mathbb R.
\end{equation}
Since \(G\) has degree \(n\), inequality \eqref{eq:bernstein} with
\(m=1\) gives \(|d|\leq\norm{G'}_\infty\leq n\).  The polynomial
\(E=G'-\mu G\) has degree at most \(n\), so
\eqref{eq:coefficient-bound} and \eqref{eq:residual-invariance} yield
\begin{equation}\label{eq:mu-d}
 |\mu-d|=|\mu G(1)-G'(1)|=|E(1)|
 \leq\sqrt{n+1}\,\norm E_2
 \leq\sqrt{n+1}\,\e^{-c_n n}=o(1).
\end{equation}
In particular, \(|\mu-d|<1\) for all sufficiently large \(n\).

\subsection{Separation of the zeros}\label{subsec:separation}

Fix a zero \(\beta\) of \(G\), and write
\(\delta=1-\RePart\beta\geq0\).  Integrating the exact identity
\[
 (\e^{-\mu w}G(w))'=\e^{-\mu w}E(w)
\]
along the segment from \(\beta\) to \(1\), and using \(G(\beta)=0\)
and \(G(1)=1\), gives
\begin{equation}\label{eq:ode-integral}
 1=\int_\beta^1\e^{\mu(1-w)}E(w)\,dw.
\end{equation}
Parametrize this segment by
\(w=\beta+u(1-\beta)\), \(0\leq u\leq1\), so that
\(1-w=(1-u)(1-\beta)\).  Since \(d\in\mathbb R\) and
\(\RePart(1-\beta)=\delta\),
\[
 \RePart\bigl(\mu(1-w)\bigr)
 =(1-u)\bigl[d\delta+\RePart\bigl((\mu-d)(1-\beta)\bigr)\bigr]
 \leq(1-u)\bigl(|d|\delta+|\mu-d|\,|1-\beta|\bigr),
\]
and hence, by \(|d|\leq n\), \(|\mu-d|<1\), \(|1-\beta|\leq2\), and
\(\delta\geq0\),
\begin{equation}\label{eq:exponent}
 \RePart\bigl(\mu(1-w)\bigr)
 \leq(1-u)(n\delta+2)
 \leq n\delta+2.
\end{equation}
The segment lies in \(\overline\D\), so
\eqref{eq:coefficient-bound} gives
\(|E(w)|\leq\sqrt{n+1}\norm E_2\) on it.  Hence
\eqref{eq:ode-integral}, \eqref{eq:exponent}, and
\eqref{eq:residual-invariance} imply
\[
 1\leq|1-\beta|\,\e^{n\delta+2}\sqrt{n+1}\,\norm E_2
 \leq2\e^{n\delta+2}\sqrt{n+1}\,\e^{-c_n n}.
\]
Taking logarithms and dividing by \(n\),
\begin{equation}\label{eq:separation}
 \delta\geq
 c_n-\frac{2+\log(2\sqrt{n+1})}{n}.
\end{equation}
The subtracted term is \(O(n^{-1}\log n)=o(c_n)\), because
\(c_n\geq n^{-1/2}\log n\).  Consequently, for all sufficiently large
\(n\), every zero satisfies
\begin{equation}\label{eq:separation-half}
 1-\RePart\beta_\ell\geq\frac{c_n}{2}
 \qquad(1\leq\ell\leq n).
\end{equation}

Set
\[
 v_\ell=\frac1{1-\beta_\ell},
 \qquad
 m_k=\frac1n\sum_{\ell=1}^n v_\ell^k.
\]
Because \(|\beta_\ell|\leq1\), we have
\(2(1-\RePart\beta_\ell)-|1-\beta_\ell|^2=1-|\beta_\ell|^2\geq0\),
whence
\begin{equation}\label{eq:v-geometry}
 \RePart v_\ell
 =\frac{1-\RePart\beta_\ell}{|1-\beta_\ell|^2}
 \geq\frac12;
 \qquad
 \frac{|v_\ell|^2}{\RePart v_\ell}
 =\frac1{1-\RePart\beta_\ell}
 \leq\frac2{c_n},
\end{equation}
the second part by \eqref{eq:separation-half}.  Since
\(G(1)=1\ne0\), in a neighborhood of \(1\) the logarithmic derivative
of \(G\) is \(\sum_\ell(w-\beta_\ell)^{-1}\); evaluating at \(w=1\)
gives \(\sum_\ell v_\ell=G'(1)=d\).  Hence \(m_1=d/n\) is real by
\eqref{eq:d-real}, and, averaging the first part of
\eqref{eq:v-geometry},
\begin{equation}\label{eq:first-moment}
 m_1=\frac dn=\frac1n\sum_{\ell=1}^n\RePart v_\ell\geq\frac12.
\end{equation}

\subsection{Higher reciprocal-zero moments}\label{subsec:moments}

In a neighborhood of \(1\), define
\[
 F(w)=\frac{G'(w)}{G(w)}-\mu
 =\sum_{\ell=1}^n\frac1{w-\beta_\ell}-\mu
 =\frac{E(w)}{G(w)}.
\]
Since \(E=FG\), Leibniz's rule gives
\begin{equation}\label{eq:leibniz}
 E^{(m)}(1)=
 \sum_{r=0}^m\binom mr F^{(r)}(1)G^{(m-r)}(1).
\end{equation}
The coefficient bound \eqref{eq:coefficient-bound} gives
\(\norm E_\infty\leq\sqrt{n+1}\norm E_2\); combining this with
Bernstein's inequality \eqref{eq:bernstein}, the normalization
\eqref{eq:normG}, and the residual bound
\eqref{eq:residual-invariance} yields
\begin{equation}\label{eq:derivative-bounds}
 |E^{(m)}(1)|\leq n^m\sqrt{n+1}\,\e^{-c_n n},
 \qquad
 |G^{(m)}(1)|\leq n^m
 \qquad(0\leq m\leq n).
\end{equation}

Define
\[
 A_0=1,
 \qquad
 A_m=1+\sum_{r=0}^{m-1}\binom mr A_r
 \quad(m\geq1).
\]
We claim that
\begin{equation}\label{eq:Fderivative}
 |F^{(m)}(1)|
 \leq A_m n^m\sqrt{n+1}\,\e^{-c_n n}
 \qquad(0\leq m\leq n).
\end{equation}
For \(m=0\) this is \eqref{eq:derivative-bounds}, because
\(F(1)=E(1)\) by \(G(1)=1\).  For \(m\geq1\), solving
\eqref{eq:leibniz} for the term \(r=m\) (whose factor is
\(G(1)=1\)) and inserting \eqref{eq:derivative-bounds} and the
inductive bounds give
\[
 |F^{(m)}(1)|
 \leq|E^{(m)}(1)|
 +\sum_{r=0}^{m-1}\binom mr|F^{(r)}(1)|\,|G^{(m-r)}(1)|
 \leq\left(1+\sum_{r=0}^{m-1}\binom mr A_r\right)
 n^m\sqrt{n+1}\,\e^{-c_n n},
\]
which is \eqref{eq:Fderivative}.

By Lemma~\ref{lem:Am} in Appendix~\ref{app:asymptotics},
\begin{equation}\label{eq:Am-bound}
 A_m\leq\frac{2\,m!}{a^{m+1}}
 \qquad(m\geq0).
\end{equation}

For \(2\leq k\leq n+1\), differentiating the partial-fraction
formula for \(F\) \(k-1\) times eliminates the constant \(-\mu\) and
gives
\[
 F^{(k-1)}(1)=(-1)^{k-1}(k-1)!
 \sum_{\ell=1}^n v_\ell^k.
\]
Combining this identity with \eqref{eq:Fderivative} and
\eqref{eq:Am-bound} yields
\begin{equation}\label{eq:moment-bound}
 |m_k|
 =\frac{|F^{(k-1)}(1)|}{n(k-1)!}
 \leq\frac{2}{a^k}
 n^{k-2}\sqrt{n+1}\,\e^{-c_n n}
 \qquad(2\leq k\leq n+1).
\end{equation}

\subsection{The optimized polynomial test}\label{subsec:test}

With \(c_n\) as defined above, put
\[
 y=c_n n,
 \qquad
 H=\log\Bigl(1+\frac y{2a}\Bigr),
 \qquad
 K=\left\lfloor\frac{y-2\log n}{H}\right\rfloor,
 \qquad
 \vartheta=\sqrt{1-\frac{c_n}4}.
\]
The dependence of \(y,H,K,\vartheta\) on \(n\) is suppressed.
By Lemma~\ref{lem:test-asymptotics} in
Appendix~\ref{app:asymptotics}, \(K\sim2\sqrt n\)---so that
\(2\leq K\leq n\) for all sufficiently large \(n\)---and
\begin{equation}\label{eq:contraction-value}
 \frac2{\sqrt{c_n}}\,\vartheta^{K-1}
 \longrightarrow\e^{-\epsilon/2}<1
 \qquad(n\to\infty).
\end{equation}
Consider
\[
 \Phi(v)=v\left(1-\frac{c_n}{2}v\right)^{K-1}.
\]
By \eqref{eq:v-geometry},
\[
 \left|1-\frac{c_n}{2}v_\ell\right|^2
 =1-c_n\RePart v_\ell+\frac{c_n^2}{4}|v_\ell|^2
 \leq1-c_n\RePart v_\ell+\frac{c_n}{2}\RePart v_\ell
 =1-\frac{c_n}{2}\RePart v_\ell
 \leq1-\frac{c_n}{4}.
\]
Hence
\(|\Phi(v_\ell)|\leq|v_\ell|\vartheta^{K-1}\) for every
\(\ell\).  By the Cauchy--Schwarz inequality, the second part of
\eqref{eq:v-geometry}, and \eqref{eq:first-moment},
\[
 \left|\frac1n\sum_{\ell=1}^n\Phi(v_\ell)\right|
 \leq\vartheta^{K-1}
 \left(\frac1n\sum_{\ell=1}^n|v_\ell|^2\right)^{1/2}
 \leq\vartheta^{K-1}
 \left(\frac2{c_n}\,m_1\right)^{1/2}.
\]
Dividing by \(m_1\) and using \(m_1\geq\tfrac12\),
\begin{equation}\label{eq:test-contraction}
 \frac1{m_1}\left|
 \frac1n\sum_{\ell=1}^n\Phi(v_\ell)
 \right|
 \leq\left(\frac2{c_nm_1}\right)^{1/2}\vartheta^{K-1}
 \leq\frac2{\sqrt{c_n}}\,\vartheta^{K-1}.
\end{equation}

Next, write
\[
 \Phi(v)=v+\sum_{k=2}^K \lambda_kv^k,
 \qquad
 \lambda_k=(-1)^{k-1}\binom{K-1}{k-1}
 \left(\frac{c_n}{2}\right)^{k-1}.
\]
Since \(c_nn=y\) and
\((c_n/2)^{k-1}(1/a)^{k-1}n^{k-1}=(y/(2a))^{k-1}\),
the moment estimate \eqref{eq:moment-bound} gives
\begin{equation}\label{eq:weighted-moments}
\begin{aligned}
 \left|\sum_{k=2}^K \lambda_km_k\right|
 &\leq\frac2a\frac{\sqrt{n+1}}n\e^{-y}
 \sum_{k=2}^K\binom{K-1}{k-1}\left(\frac y{2a}\right)^{k-1}\\
 &\leq\frac2a\frac{\sqrt{n+1}}n\e^{-y}
 \left(1+\frac y{2a}\right)^{K-1}\\
 &\leq\frac2a\frac{\sqrt{n+1}}{n^3}
 \longrightarrow0,
\end{aligned}
\end{equation}
where the middle inequality is the binomial theorem and the last one
follows from \((K-1)H\leq y-2\log n\), that is,
\((1+y/(2a))^{K-1}\leq\e^{y}n^{-2}\).

Averaging the binomial expansion of \(\Phi\) over
\(v_1,\ldots,v_n\) gives
\[
 \frac1n\sum_{\ell=1}^n\Phi(v_\ell)
 =m_1+\sum_{k=2}^K \lambda_km_k.
\]
By \eqref{eq:first-moment} and \eqref{eq:weighted-moments},
\[
 \frac1{m_1}\left|\frac1n\sum_{\ell=1}^n\Phi(v_\ell)\right|
 \geq1-\frac1{m_1}\left|\sum_{k=2}^K \lambda_km_k\right|
 \geq1-2\left|\sum_{k=2}^K \lambda_km_k\right|
 \longrightarrow1.
\]
This contradicts \eqref{eq:test-contraction} and
\eqref{eq:contraction-value}, which force
\[
 \limsup_{n\to\infty}
 \frac1{m_1}\left|
 \frac1n\sum_{\ell=1}^n\Phi(v_\ell)
 \right|
 \leq\e^{-\epsilon/2}<1.
\]
Hence no such sequence exists, so for all sufficiently large \(n\),
every \(P\in\Pcal_n\) satisfies
\[
 \rho(P)\geq
 \exp\!\left[-\sqrt n\left(
 \log n+\log\frac2{\log2}+\epsilon\right)\right].
\]
Taking the infimum over \(P\in\Pcal_n\) proves
Theorem~\ref{thm:stable}.
\end{proof}

The residual infimum \(r_n\) also has a trivial upper bound, and the
two bounds together determine its exponential scale.

\begin{corollary}\label{cor:root-asymptotic}
\(r_n\leq n/\sqrt2\) for every \(n\geq2\), and
\[
 r_n^{1/n}\longrightarrow1
 \qquad(n\to\infty).
\]
\end{corollary}

\begin{proof}
For \(P(t)=1-t^n\) we have \(P\in\Pcal_n\), \(P'(0)=0\), and
\[
 \rho(P)=\frac{\norm{P'}_2}{\norm P_2}=\frac n{\sqrt2},
\]
so \(r_n\leq n/\sqrt2\), and hence
\(\limsup_{n\to\infty}r_n^{1/n}\leq1\).  Theorem~\ref{thm:stable}
gives \(\liminf_{n\to\infty}r_n^{1/n}\geq1\).
\end{proof}

\section{Application to exterior power sums}\label{sec:applications}

For \(z=(z_1,\ldots,z_n)\in\C^n\), put
\[
 p_k(z)=\sum_{j=1}^n z_j^k,
 \qquad
 M_n(z)=\max_{2\leq k\leq n+1}|p_k(z)|.
\]
The link between power sums and the residual is the following exact
identity.

\begin{proposition}[Projected Newton identity]\label{prop:reduction}
Let \(z_1,\ldots,z_n\in\C\) satisfy \(|z_j|\geq1\), and set
\(P(t)=\prod_{j=1}^n(1-z_jt)\).
Then \(P\in\Pcal_n\) and
\[
 \rho(P)\leq nM_n(z).
\]
\end{proposition}

\begin{proof}
The zeros of \(P\) are \(1/z_j\in\overline\D\), and \(P(0)=1\), so
\(P\in\Pcal_n\).  Abbreviate \(p_k=p_k(z)\), and put
\[
 R=\left(\max_{1\leq j\leq n}|z_j|\right)^{-1}>0.
\]
For \(|t|<R\), we have \(P(t)\neq0\), and logarithmic
differentiation together with the finite geometric identity
\[
 \frac{z_j}{1-z_jt}
 =\sum_{k=1}^{n+1}z_j^kt^{k-1}
 +\frac{z_j^{n+2}t^{n+1}}{1-z_jt}
\]
gives
\[
 \frac{P'(t)}{P(t)}
 =-\sum_{j=1}^n\frac{z_j}{1-z_jt}
 =-\sum_{k=1}^{n+1}p_kt^{k-1}
   -t^{n+1}\sum_{j=1}^n\frac{z_j^{n+2}}{1-z_jt}.
\]
Multiplying by \(P(t)\) yields
\[
 P'(t)+P(t)\sum_{k=1}^{n+1}p_kt^{k-1}
 =-t^{n+1}\sum_{j=1}^nz_j^{n+2}
 \prod_{\substack{1\leq i\leq n\\i\ne j}}(1-z_it),
 \qquad |t|<R.
\]
Both sides are polynomials, so this identity holds for every
\(t\in\C\).  Let \(\Pi_{\leq n}\) denote coefficient projection onto
degrees at most \(n\), and set
\(S(t)=\sum_{k=2}^{n+1}p_kt^{k-1}\).  The right-hand side above is
divisible by \(t^{n+1}\), so applying \(\Pi_{\leq n}\) and using
\(P'(0)=-p_1\) give
\begin{equation}\label{eq:projection}
 P'-P'(0)P=-\Pi_{\leq n}(SP).
\end{equation}
The projection \(\Pi_{\leq n}\) does not increase the \(H^2\) norm,
and Parseval's identity gives
\(\norm{SP}_2\leq\norm S_\infty\norm P_2\).  Hence
\[
 \rho(P)
 =\frac{\norm{\Pi_{\leq n}(SP)}_2}{\norm P_2}
 \leq\frac{\norm{SP}_2}{\norm P_2}
 \leq\norm S_\infty
 \leq\sum_{k=2}^{n+1}|p_k|
 \leq nM_n(z).
\qedhere
\]
\end{proof}

Combining Proposition~\ref{prop:reduction} with
Theorem~\ref{thm:stable} gives a quantitative lower bound for
exterior power sums.

\begin{corollary}[Exterior power sums]\label{cor:power}
For every \(\epsilon>0\), there exists \(N_\epsilon\) such that, for
every integer \(n\geq N_\epsilon\) and every
\(z_1,\ldots,z_n\in\C\) satisfying \(|z_j|\geq1\)
for \(1\leq j\leq n\),
\[
 M_n(z)
 \geq\frac1n\exp\!\left[-\sqrt n\left(
 \log n+\log\frac2{\log2}+\epsilon\right)\right].
\]
In particular,
\(M_n(z)\geq\exp(-(1+o(1))\sqrt n\log n)\)
uniformly over such configurations.
\end{corollary}

\begin{proof}
Proposition~\ref{prop:reduction} and the definition of \(r_n\) give
\[
 r_n\leq\rho(P)\leq nM_n(z),
\]
so \(M_n(z)\geq r_n/n\).  The conclusion follows from
Theorem~\ref{thm:stable}.
\end{proof}

For comparison, Tur\'an's first main theorem gives, in the present
notation,
\[
 M_n(z)\geq n\left(\frac{n}{2\e(n+1)}\right)^{n-1}
 =(2\e)^{-(1+o(1))n};
\]
see \cite[pp.~17--18]{VanderPoorten1970} and \cite{Turan1984}.
Luo, Yang, and Zhu~\cite{LuoYangZhu2026}
proved that \(M_n(z)>\e^{-\lambda n}\) for every fixed \(\lambda>0\)
and all sufficiently large \(n\).  Corollary~\ref{cor:power} replaces
the linear exponent by \(\sqrt n\log n\).

\begin{proof}[Proof of Corollary~\ref{cor:973}]
Suppose such a constant \(C>1\) exists.  Apply
Corollary~\ref{cor:power} with \(\epsilon=1\): for all sufficiently
large \(n\), every configuration with \(|z_j|\geq1\)---in particular
every configuration with the additional constraint \(z_1=1\)---satisfies
\[
 M_n(z)
 \geq\frac1n\exp\!\left[-\sqrt n\left(
 \log n+\log\frac2{\log2}+1\right)\right]
 >C^{-n},
\]
the last inequality for all sufficiently large \(n\).  This
contradicts the requirement \(M_n(z)<C^{-n}\) for every \(n\geq2\).
\end{proof}

\section{Further questions}

Theorem~\ref{thm:stable} determines a subexponential lower scale for
\(r_n\), but the trivial witness \(1-t^n\) in
Corollary~\ref{cor:root-asymptotic} is far from matching it.
Determining the true order of \(r_n\), or constructing polynomials
whose residuals approach the lower scale, remains open.  Natural
variants replace \(H^2\) by \(H^p\), use weighted coefficient norms,
or constrain the zeros to a smaller disk.  The proof of
Theorem~\ref{thm:stable} is effective in principle, but we have not
computed an explicit threshold \(N_\epsilon\); making the
\(o(1)\) term explicit is a further natural problem.

Likewise, Corollary~\ref{cor:power} leaves open the finer behaviour
of the exterior power-sum maximum: the present proof does not decide
whether \(M_n(z)\) can, over configurations with \(|z_j|\geq1\),
decay to zero at all as \(n\to\infty\), or must remain bounded away
from zero.

\section*{Acknowledgments}

We thank Bo-Yong Chen, Jiawen Zhang, Weiyuan Qiu, and Jun Wang for
helpful advice concerning the publication of this work, and Hongwei Lou
for detailed comments on an earlier version of the manuscript that
improved the presentation of the main result and clarified the proof
structure.  Qihang Wang thanks his doctoral advisor, Weinan E, for his
encouragement and support.  We also acknowledge Gewu Intelligence Lab
for providing the collaborative research environment in which this
project was developed.

Kun Chen and Xiaojun Tan are supported by the Strategic Priority
Research Program of the Chinese Academy of Sciences under Grant
No.~XDB1680102.

The authors declare that they have no conflicts of interest.

\section*{Data access statement}

Data sharing is not applicable to this article, as no datasets were
generated or analysed in this work.

\section*{Disclosure of automated assistance}

OpenAI GPT-5.6 Sol was used during proof exploration and drafting.
Anthropic Claude Fable 5, operating in an agentic workflow, was used
for adversarial checking, literature comparison, and language editing.
The authors independently verified all resulting mathematical and
bibliographic claims.  The argument in this paper is self-contained,
no model output is used as a mathematical premise, and the named
authors assume full responsibility for the final manuscript, including
every statement, proof, citation, and priority claim.

\appendix

\section{Two elementary computations}\label{app:asymptotics}

This appendix contains the two elementary computations used in the
proof of Theorem~\ref{thm:stable}; neither involves the polynomial
\(P\).

\begin{lemma}\label{lem:Am}
Let \(A_0=1\), let \(A_m=1+\sum_{r=0}^{m-1}\binom mrA_r\) for
\(m\geq1\), and let \(a=\log2\).  Then
\[
 A_m=\sum_{q=1}^\infty\frac{q^m}{2^q}
 \qquad\text{and}\qquad
 A_m\leq\frac{2\,m!}{a^{m+1}}
 \qquad(m\geq0).
\]
\end{lemma}

\begin{proof}
The recursion is equivalent to
\(\sum_{r=0}^m\binom mrA_r=2A_m-1\) for every \(m\geq0\);
multiplying by \(z^m/m!\) and summing, the exponential generating
function \(f(z)=\sum_{m\geq0}A_mz^m/m!\) satisfies
\(\e^zf(z)=2f(z)-\e^z\), so that
\[
 f(z)
 =\frac{\e^z}{2-\e^z}
 =\sum_{q=1}^\infty\frac{\e^{qz}}{2^q},
 \qquad\text{and hence}\qquad
 A_m=\sum_{q=1}^\infty\frac{q^m}{2^q}.
\]
Since \(q^m2^{-q}=q^m\e^{-aq}\leq(x+1)^m\e^{-ax}\) for
\(x\in[q-1,q]\), summation over \(q\geq1\) gives
\[
 A_m
 \leq\int_0^\infty(x+1)^m\e^{-ax}\,dx
 =\sum_{r=0}^m\binom mr\frac{r!}{a^{r+1}}
 =\frac{m!}{a^{m+1}}\sum_{q=0}^m\frac{a^q}{q!}
 \leq\frac{m!}{a^{m+1}}\,\e^{a}
 =\frac{2m!}{a^{m+1}}.
 \qedhere
\]
\end{proof}

\begin{lemma}\label{lem:test-asymptotics}
Let \(a=\log2\), and let \(\sigma\in\mathbb R\) be fixed.  For each
integer \(n\geq2\), define
\[
 c_n=\frac{\log n+\sigma}{\sqrt n},
 \qquad
 y=c_nn,
 \qquad
 H=\log\Bigl(1+\frac y{2a}\Bigr),
 \qquad
 K=\left\lfloor\frac{y-2\log n}{H}\right\rfloor,
 \qquad
 \vartheta=\sqrt{1-\frac{c_n}4}.
\]
As in Section~\ref{subsec:test}, the dependence of
\(y,H,K,\vartheta\) on \(n\) is suppressed.
Then \(K\sim2\sqrt n\) and
\[
 \frac2{\sqrt{c_n}}\,\vartheta^{K-1}
 \longrightarrow
 \exp\!\left[-\left(\frac\sigma2+\frac12\log(2a)-\log2\right)\right]
 \qquad(n\to\infty).
\]
In particular, for \(\sigma=\log(2/a)+\epsilon\) the limit equals
\(\e^{-\epsilon/2}\).
\end{lemma}

\begin{proof}
Since \(y\to\infty\),
\begin{equation}\label{eq:H-expansion}
 H=\log\frac y{2a}+\log\left(1+\frac{2a}y\right)
 =\frac12\log n+\log(\log n+\sigma)-\log(2a)+o(1),
\end{equation}
so \(H=\bigl(\tfrac12+o(1)\bigr)\log n\).  As
\(y-2\log n=(1+o(1))y\) and \(y=\sqrt n(\log n+\sigma)\), this gives
\[
 K=(1+o(1))\frac yH
 =(1+o(1))\frac{\sqrt n(\log n+\sigma)}{\frac12\log n}
 \sim2\sqrt n.
\]

For the limit, take logarithms: with
\[
 \gamma=-\log\vartheta=-\frac12\log(1-c_n/4)
 =\frac{c_n}8+O(c_n^2),
\]
it suffices to show
\(\gamma(K-1)-\log(2/\sqrt{c_n})
\to\frac\sigma2+\frac12\log(2a)-\log2\).
Write \(H=\frac12\log n+h_n\) with
\(h_n=\log(\log n+\sigma)-\log(2a)+o(1)=O(\log\log n)\), by
\eqref{eq:H-expansion}.  Then
\[
\begin{aligned}
 \frac{(\log n+\sigma)^2}{8H}
 &=\frac{(\log n+\sigma)^2}{4\log n}
 \left(1-\frac{2h_n}{\log n}
 +O\!\left(\Bigl(\frac{h_n}{\log n}\Bigr)^{2}\right)\right)\\
 &=\frac14\log n+\frac\sigma2-\frac{h_n}2+o(1)\\
 &=\frac14\log n-\frac12\log\log n
 +\frac\sigma2+\frac12\log(2a)+o(1).
\end{aligned}
\]
Moreover, \(K=(y-2\log n)/H+O(1)\), and \(K=O(\sqrt n)\) gives
\(c_n^2K=o(1)\).  Since \(c_ny=(\log n+\sigma)^2\) exactly,
\[
 c_n(y-2\log n)
 =(\log n+\sigma)^2-\frac{2(\log n+\sigma)\log n}{\sqrt n}
 =(\log n+\sigma)^2+o(1).
\]
Hence
\[
 \gamma(K-1)
 =\frac{c_n(y-2\log n)}{8H}+O(c_n)+O(c_n^2K)
 =\frac{(\log n+\sigma)^2}{8H}+o(1),
\]
while
\[
 \log\frac2{\sqrt{c_n}}
 =\log2-\frac12\log c_n
 =\log2+\frac14\log n-\frac12\log\log n+o(1).
\]
Subtracting, the divergent terms \(\frac14\log n\) and
\(-\frac12\log\log n\) cancel exactly, and
\[
 \gamma(K-1)-\log\frac2{\sqrt{c_n}}
 =\frac\sigma2+\frac12\log(2a)-\log2+o(1),
\]
which proves the stated limit.  Finally, for
\(\sigma=\log(2/a)+\epsilon\),
\[
 \frac\sigma2+\frac12\log(2a)-\log2
 =\frac12\log\frac2a+\frac\epsilon2
 +\frac12\log(2a)-\log2
 =\frac\epsilon2.
 \qedhere
\]
\end{proof}


\begin{thebibliography}{99}

\bibitem{Bloom973}
T.~F. Bloom,
\emph{Erd\H{o}s Problem \#973},
Erd\H{o}s Problems,
\url{https://www.erdosproblems.com/973},
accessed 31 July 2026.

\bibitem{Erdos1965}
P.~Erd\H{o}s,
Some recent advances and current problems in number theory,
in \emph{Lectures on Modern Mathematics}, Vol.~III,
T.~L. Saaty, ed.,
Wiley, New York, 1965, pp.~196--244.

\bibitem{Hayman1974}
W.~K. Hayman,
Research problems in function theory: new problems,
in \emph{Proceedings of the Symposium on Complex Analysis
(Canterbury, 1973)},
London Math. Soc. Lecture Note Ser., vol.~12,
Cambridge Univ. Press, London, 1974, pp.~155--180,
\url{https://doi.org/10.1017/CBO9780511662263.034}.

\bibitem{LuoYangZhu2026}
Y.~Luo, R.~Yang, and K.~Zhu,
Exterior power sums,
arXiv:2607.22017v1 [math.CO], 24 July 2026,
\url{https://arxiv.org/abs/2607.22017}.

\bibitem{Turan1984}
P.~Tur\'an,
\emph{On a New Method of Analysis and Its Applications},
with the assistance of G.~Hal\'asz and J.~Pintz,
Pure and Applied Mathematics,
Wiley-Interscience, New York, 1984.

\bibitem{VanderPoorten1970}
A.~J. van der Poorten,
Generalisations of Tur\'an's main theorems on lower bounds for sums
of powers,
\emph{Bull. Austral. Math. Soc.} \textbf{2} (1970), 15--37,
\url{https://doi.org/10.1017/S0004972700041575}.

\end{thebibliography}
\end{document}